\documentclass[12pt,a4paper]{amsart} 

\usepackage{latexsym}
\usepackage{amsmath}
\usepackage{amsthm}
\usepackage{tikz}
\usepackage{latexsym}
\usepackage{amssymb}
\usepackage[a4paper , lmargin = {2cm} , rmargin = {2cm} , tmargin = {2.5cm} , bmargin = {2.5cm} ]{geometry}

\usepackage{pinlabel}                    
\usepackage{graphicx}
\usepackage{epstopdf}
\usepackage{romannum}
\usepackage{tikz}
\usepackage{subfig}
\usepackage{MnSymbol}
\usepackage[arrow, matrix, curve]{xy}

\theoremstyle{plain}

\newtheorem{theorem}{Theorem}[section]
\newtheorem{lemma}[theorem]{Lemma}

\newtheorem{corollary}[theorem]{Corollary}
\newtheorem{conjecture}[theorem]{Conjecture}

\newtheorem{proposition}[theorem]{Proposition}
\theoremstyle{definition}

\newtheorem{definition}[theorem]{Definition}
\usepackage{graphicx}
\newcommand{\Z}{\ensuremath{\mathbb{Z}}}

\title[]{Coxeter quotients of the automorphism group of a Coxeter group}

\author{Olga Varghese}
\date{\today}

\address{Olga Varghese\\
	Department of Mathematics\\
	Otto-von-Guericke University of Magdeburg\\ 
	Universit\"atsplatz 2\\
	39106 Magdeburg (Germany)}
\email{olga.varghese@ovgu.de}

\thanks{The work was funded by the Deutsche Forschungsgemeinschaft (DFG, German Research Foundation) under Germany's Excellence Strategy EXC 2044--390685587, Mathematics M\"unster: Dynamics-Geometry-Structure and DFG grant VA 1397/2-1.}

\begin{document}
\pagenumbering{arabic}

\newpage
\begin{abstract}
We show that for a large class $\mathcal{W}$ of Coxeter groups the following holds: Given a group $W_\Gamma$ in $\mathcal{W}$, 
the automorphism group  ${\rm Aut}(W_\Gamma)$ virtually surjects onto some infinite Coxeter group. In particular, the group ${\rm Aut}(W_\Gamma)$ is virtually indicable and therefore does not have Kazhdan's property (T).
\end{abstract}
\maketitle

\section{Introduction}
One fascinating property of a group is property (T). It was defined by Kazhdan for topological groups in terms of unitary representations and was reformulated  by Delorme and Guichardet in geometric group theory. A countable group $G$ has Kazhdan's property (T) if  every action of $G$ on a real Hilbert space by isometries has a global fixed point (\cite[Theorem 2.12.4]{BekkaHarpeValette}). Examples of groups satisfying this property are finite groups (\cite[Proposition 1.1.5]{BekkaHarpeValette}), the general linear group ${\rm GL}_n(\mathbb{Z})$ for $n\geq 3$ (\cite[Theorem 4.2.5]{BekkaHarpeValette}), the automorphism group of a free group ${\rm Aut}(F_n)$ for $n\geq 4$ (\cite{KalubaKielakNowak}, \cite{KalubaNowakOzawa}, \cite{Nitsche}). In contrast, the groups ${\rm Aut}(F_2)$ and ${\rm GL}_2(\Z)$ do not have property (T). One algebraic reason for ${\rm Aut}(F_2)$ and ${\rm GL}_2(\Z)$ to not have property (T) is the fact that both groups can be written as amalgamated products. 

It is natural to try to relate properties of a group $G$ to properties of ${\rm Aut}(G)$. We investigate the following question:
\vspace{0.3cm}
\begin{quote}
\em{Under which conditions on the group $G$ does the automorphism group ${\rm Aut}(G)$ not satisfy Kazhdan's property (T)}? 
\end{quote}
\vspace{0.3cm}
In particular, we are interested in algebraic properties of ${\rm Aut}(G)$ that prohibit this group from having property (T). Here we focus on groups which are defined in a combinatorial way. Given a finite simplicial graph $\Gamma$ with the vertex set $V(\Gamma)$, the edge set $E(\Gamma)$ and with an edge-labeling $\varphi\colon E(\Gamma)\rightarrow\mathbb{N}_{\geq 2}$,  the Coxeter group $W_\Gamma$ associated to $\Gamma$ is the group with the presentation 
$$W_\Gamma=\langle V\mid v^2, (vw)^{\varphi(\left\{v,w\right\})} \text{ for all }v\in V(\Gamma) \text{ and whenever } \left\{v,w\right\}\in E(\Gamma)\rangle.$$
 A special subclass of the class of all Coxeter groups is the class consisting of right-angled Coxeter groups. Recall, a Coxeter group $W_\Gamma$ is called \emph{right-angled} if $\Gamma$ is discrete or $\varphi(E(\Gamma))=\left\{2\right\}$.

Let us consider the graphs and the associated Coxeter groups in Figure 1. 
\begin{figure}[h]
\begin{center}
\begin{tikzpicture}[scale=0.90]
\draw[fill=black]  (0,0) circle (2pt);
\draw[fill=black]  (2,0) circle (2pt);
\draw[fill=black]  (2,2) circle (2pt);
\draw[fill=black]  (0,2) circle (2pt);
\node at (1,-0.4) {$\Gamma_1$}; 

\draw[fill=black]  (5,0) circle (2pt);
\draw[fill=black]  (7,0) circle (2pt);
\draw[fill=black]  (5,2) circle (2pt);
\draw[fill=black]  (7,2) circle (2pt);
\draw (5,2)--(7,2);
\node at (6, 2.25) {$3$}; 
\node at (6,-0.4) {$\Gamma_2$};  

\draw[fill=black]  (10,0) circle (2pt);
\draw[fill=black]  (12,0) circle (2pt);
\draw[fill=black]  (10,2) circle (2pt);
\draw[fill=black]  (12,2) circle (2pt);
\draw (10,0)--(12,0);
\draw (10,0)--(10,2);
\draw (10,0)--(12,2);
\node at (11,0.2) {$2$}; 
\node at (9.8, 1) {$2$}; 
\node at (11, 1.3) {$2$}; 
\node at (11,-0.4) {$\Gamma_3$}; 
\end{tikzpicture}
\caption{Examples of Coxeter graphs.}
\end{center}
\end{figure}
The Coxeter group $W_{\Gamma_1}$ is isomorphic to $\Z/2\Z*\Z/2\Z*\Z/2\Z*\Z/2\Z$, $W_{\Gamma_2}\cong\Z/2\Z*\Z/2\Z*{\rm Sym}(3)$ and $W_{\Gamma_3}$ is isomorphic to the direct product of $\Z/2\Z$ and $\Z/2\Z*\Z/2\Z*\Z/2\Z$. We note that if $\Gamma$ is disconnected with connected components $\Gamma_1,\ldots,\Gamma_n$, then $W_\Gamma$ is the free product $W_{\Gamma_1}*\ldots*W_{\Gamma_n}$ and if $\Gamma$ is a join $\Gamma=\Gamma_1*\Gamma_2$ and $\varphi(\left\{v,w\right\})=2$ for all $v\in V(\Gamma_1)$ and $w\in V(\Gamma_2)$, then $W_\Gamma$ is the direct product $W_{\Gamma_1}\times W_{\Gamma_2}$. We call $W_\Gamma$ \emph{irreducible} if $W_\Gamma$ does not admit such a direct product decomposition.

The emphasis of this article is mainly on Kazhdan's property (T). It was proven in \cite{BJS} that  a Coxeter group $W_\Gamma$ has Kazhdan's property (T) if and only if $W_\Gamma$ is finite. A characterization of finite irreducible Coxeter groups in terms of Dynkin diagrams is given in \cite[\S 2]{Humphreys}. Hence, if a Coxeter group $W_\Gamma$ has property (T) or not can be checked easily on the Coxeter graph $\Gamma$.

Coxeter groups are well understood objects in geometric group theory, but there are many open questions concerning their automorphism groups.  We address the following conjecture that was formulated by Alain Valette during the virtual geometric group theory conference in Luminy, 2020.

\begin{conjecture}
	\label{Con1}
For every infinite Coxeter group $W_\Gamma$, the automorphism group ${\rm Aut}(W_\Gamma)$ virtually maps onto some infinite Coxeter group. 
\end{conjecture}
\vspace{0.2cm}
\begin{quote}
	\em{We denote by $\mathcal{W}$ the class of Coxeter groups that satisfy the above conjecture.}
\end{quote}
\vspace{0.3cm}

Our goal here is to verify Conjecture \ref{Con1} for a large class of Coxeter groups. But before we move on, we want to make a note on some consequences for a Coxeter group to be in the class $\mathcal{W}$.

Recall, a group $G$ is said to be \emph{virtually indicable} if there exists a finite index subgroup $H\subseteq G$ such that $H$ surjects onto $\Z$. It was proven in \cite{Gonciulea} that an infinite Coxeter group is virtually indicable. Note that Kazhdan's property (T) is preserved under taking finite index subgroups (\cite[Proposition 2.5.7]{BekkaHarpeValette}) and is inherited by quotients (\cite[Proposition 2.5.1]{BekkaHarpeValette}). Since $\Z$ does not have property (T) we obtain the following result.
\begin{corollary}
	\label{T}
	Let $W_\Gamma$ be a Coxeter group. If $W_\Gamma$ is in the class $\mathcal{W}$, then ${\rm Aut}(W_\Gamma)$ is virtually indicable and therefore does not have Kazhdan's property (T).
\end{corollary}

In particular, we conjecture

\begin{conjecture}
	Every automorphism group ${\rm Aut}(W_\Gamma)$ of an infinite Coxeter group $W_\Gamma$  does not have Kazhdan's property (T).
\end{conjecture}

Results regarding Conjecture \ref{Con1} were proved in \cite[Corollary 1.3]{GenevoisVarghese} where it is shown that the automorphism group of an infinite right-angled Coxeter group virtually maps onto $\Z/2\Z*\Z/2\Z$. In particular, all infinite right-angled Coxeter groups are in the class $\mathcal{W}$. The special case where $\Gamma$ has no edges was shown in \cite[Theorem 3.5]{Varghese}.

\subsection{Finite outer automorphism groups of Coxeter groups}
First we consider cases, where the subgroup ${\rm Inn}(W_\Gamma)$ consisting of inner automorphisms of $W_\Gamma$ has finite index in ${\rm Aut}(W_\Gamma)$. 
\begin{lemma}
	\label{FiniteOut}
	Let $W_\Gamma$ be a Coxeter group. If $W_\Gamma$ is infinite and ${\rm Out}(W_\Gamma)$ is finite, then ${\rm Aut}(W_\Gamma)$ virtually maps onto an infinite special subgroup of $W_\Gamma$.  
\end{lemma}

Examples of Coxeter groups with finite outer automorphism group are Coxeter groups $W_\Gamma$ where $\Gamma$ is a complete graph (\cite{HowlettRowleyTaylor}), affine Coxeter groups (\cite[Proposition 4.5]{Franzsen}), hyperbolic Coxeter groups in the sence of Humphreys \cite[Theorem 4.10]{Franzsen} and many right-angled Coxeter groups \cite{SaleSusse}. 

Recall, given a graph $\Gamma$, a {\em maximal  clique} $\Delta$ in $\Gamma$ is a maximal complete subgraph of $\Gamma$.

\begin{theorem}
	\label{InfiniteComplete}
	Let $\Gamma$ be a connected graph and $W_\Gamma$ be a Coxeter group. We denote by $\Delta_1, \ldots, \Delta_n$ its maximal cliques and by $\Omega_1,\ldots, \Omega_m$ the nonempty intersections of these maximal cliques where the intersection is taken over at least two maximal cliques. 
	
	If the centralizer of the special parabolic subgroup $W_{\Omega_i}$ is trivial for all $i=1,\ldots, m$, then the subgroup ${\rm Inn}(W_\Gamma)$ has finite index in ${\rm Aut}(W_\Gamma)$.  
\end{theorem}

Before we move on, let us consider an example of a Coxeter group that satisfies the assumptions of Theorem \ref{InfiniteComplete}. 

\begin{figure}[h]
	\begin{center}
		\begin{tikzpicture}
			\draw[fill=black]  (0,0) circle (2pt);
			\draw[fill=black]  (2,1) circle (2pt);			
			\draw[fill=black]  (0,2) circle (2pt);
			\draw (0,0)--(2,1);
			\draw (0,2)--(2,1);
			\node at (1, 0.2) {$3$}; 
			\draw (0,2)--(0,0);
			\node at (-0.2, 1) {$3$};
			\draw[fill=black]  (-2,1) circle (2pt);
			\draw (-2,1)--(0,2);
			\node at (-1, 1.8) {$7$};
			\node at (1, 1.8) {$3$};
			\draw (-2,1)--(0,0);
			\node at (-1, 0.2) {$3$};
		\end{tikzpicture}
	\caption{}
	\end{center}
\end{figure}

The graph in Figure $2$ has two maximal cliques $\Delta_1$ and $\Delta_2$. Further the centralizer of the special parabolic subgroup $W_{\Delta_1\cap\Delta_2}=\langle v,w\mid v^2=1, w^2=1, (vw)^3=1\rangle\cong{\rm Sym}(3)$ is trivial. Hence, by Theorem \ref{InfiniteComplete} the group ${\rm Out}(W_\Gamma)$ is finite. We note that the structure of centralizers of special parabolic subgroups of Coxeter groups was studied in \cite{Brink}, \cite{BahlsIso}, \cite{Nuida1}.
  
\subsection{Automorphism groups of Coxeter groups with disconnected graphs} 
Our next result is about ${\rm Aut}(W_\Gamma)$ where $\Gamma$ has at least two connected components. Using Kurosh's subgroup theorem \cite{Kurosh} we reduce Conjecture \ref{Con1} to connected graphs.

\begin{theorem}
	\label{TwoConnectedComponents}
	Let $W_\Gamma$ be a Coxeter group and let $\Gamma_1,\ldots, \Gamma_n$ be the connected components of $\Gamma$. If $n\geq2$, then for all $i, j\in\left\{1,\ldots, n\right\}$, $i\neq j$ the automorphism group ${\rm Aut}(W_\Gamma)$ virtually
	surjects onto $W^{ab}_{\Gamma_i}*W^{ab}_{\Gamma_j}\cong (\Z/2\Z)^k*(\Z/2\Z)^l$ where $k, l\geq 1$.  
\end{theorem}

It is worth pointing out that the abelianisation $W^{ab}_\Gamma$ of a Coxeter group $W_\Gamma$ can be easily read off the defining graph $\Gamma$ as follows: Let $\sim$ be the equivalence relation
on $V(\Gamma)$ defined by taking the transitive closure of the relation: $v\sim w$ if and only if $v$ and $w$ are adjacent and the edge label of $\left\{v,w\right\}$ is odd. We denote by $V(\Gamma)/\sim$ the set of the equivalence classes of this equivalence relation. Then $W^{ab}_\Gamma\cong (\Z/2\Z)^k$ where $k=|V(\Gamma)/\sim|$, see \cite[Fact 3.16]{Bernhard}. 


\subsection{Automorphism groups of (even) Coxeter groups}
Given a Coxeter group $W_\Gamma$ and a special parabolic subgroup $W_\Omega$ that sits 'nicely' in $W_\Gamma$, then there exists a retraction map $W_\Gamma\twoheadrightarrow W_\Omega$ and the kernel of this map is characteristic with respect to a finite index subgroup of ${\rm Aut}(W_\Gamma)$. Using these observations we prove the following result. 
\begin{theorem}
	\label{Even}
Let $W_\Gamma$ be a Coxeter group. If there exists an infinite special subgroup $W_\Omega$ such that ${\rm Out}(W_\Omega)$ is finite and for any $\left\{v,w\right\}\in E(\Gamma)$ where $v\in V(\Gamma)-V(\Omega)$, $w\in V(\Omega)$ the edge label of $\left\{v,w\right\}$ is even,
then the automorphism group ${\rm Aut}(W_\Gamma)$  virtually surjects onto  an infinite special subgroup of $W_\Omega$. 
\end{theorem}

Theorem \ref{Even} has an immediate consequence for even Coxeter groups. Even Coxeter group are known  as a generalization of right-angled Coxeter groups. A Coxeter group $W_{\Gamma}$ is called \emph{even}, if all edge labels are even. A vertex $v\in V(\Gamma)$ is called \emph{even}, if all edge-labels of $e\in E(\Gamma)$ with $v\in e$ are even.

\begin{corollary}
	\label{CorEven}
Let $W_\Gamma$ be a Coxeter group. If there exist two non-adjacent even vertices $v, w\in V(\Gamma)$, then the automorphism group ${\rm Aut}(W_\Gamma)$  virtually surjects onto $\Z/2\Z*\Z/2\Z$.

In particular, all infinite even Coxeter groups are in the class $\mathcal{W}$.	
\end{corollary}	

\subsection{Kazhdan's property (T)}
Regarding property (T) we immediately obtain  
\begin{corollary}
Let $W_\Gamma$ be an infinite Coxeter group. If
\begin{enumerate}
	\item ${\rm Out}(W_\Gamma)$ is finite or
	\item $\Gamma$ is disconnected or
	\item there exist two non-adjacent even vertices in $\Gamma$ or
	\item there exists an infinite special parabolic subgroup $W_\Omega$  such that 
	\begin{enumerate}
		\item ${\rm Out}(W_\Omega)$ is finite and 
		\item all edges in $E(\Gamma)$ where one vertex is in $V(\Omega)$ and the other vertex is in $V(\Gamma)-V(\Omega)$ have even labels,
\end{enumerate}
then ${\rm Aut}(W_\Gamma)$ does not have Kazhdan's property (T).
\end{enumerate}	
\end{corollary}

\subsection*{Acknowledgment}
The author thanks Anthony Genevois and Philip Möller for giving many constructive comments which helped improving the quality of the manuscript. The author also thanks the referee for careful reading of the manuscript and for many helpful comments. Additionally the author thanks Luis Paris for pointing out a gap in the previous version of this article.
\section{Graphs and groups}

\subsection{Graphs}
In this section we define simplicial graphs and collect some important properties of these combinatorial objects which we will need to define the groups of our interest. We largely follow \cite{Diestel}.
 
A \emph{simplicial graph} $\Gamma$ consists of a set $V(\Gamma)$ of \emph{vertices}
and a set $E(\Gamma)$ of subsets
of $V(\Gamma)$ of cardinality two which are called \emph{edges}. A graph $\Gamma$ is said to be \emph{finite} if $V(\Gamma)$ is a finite set. Two vertices $v, w\in V(\Gamma)$ are called \emph{adjacent}, if $\left\{v, w\right\}$ is an edge of $\Gamma$. By an \emph{edge-labeling} of $\Gamma$ we mean just a map $\varphi\colon E(\Gamma)\to \left\{2,3,\ldots\right\}$.
In practice, a graph $\Gamma$ is often represented by a diagram. We draw a point for each vertex $v\in V(\Gamma)$ of the graph, and a line joining two vertices $v$ and $w$ if $\left\{v, w\right\}\in E(\Gamma)$. If all
the vertices of $\Gamma$ are pairwise adjacent, then $\Gamma$ is called \emph{complete}.  A graph $\Lambda=(X,Y)$ is a subgraph of $\Gamma$ if $X\subseteq V(\Gamma)$, $Y\subseteq E(\Gamma)$ and all vertices of the edges in $Y$ are in $X$. A \emph{(maximal) clique} in $\Gamma$ is a (maximal) complete subgraph of $\Gamma$. Given a subset $S\subseteq V(\Gamma)$, the \emph{graph generated by $S$}, denoted by $\langle S\rangle$ is the graph with vertex set $S$
and edge set $\{\{v,w\}\in E(\Gamma)\mid v,w\in S\}$. A subgraph $\Lambda=(X,Y)$ of $\Gamma$ is called \emph{full} if $\langle X\rangle=\Lambda$. 
A graph $\Gamma$ is called \emph{connected} if for every two vertices $v, w\in V(\Gamma)$ there exist vertices $v_1,\ldots, v_n\in V(\Gamma)$ such that $\left\{v,v_1\right\}, \left\{v_i,v_{i+1}\right\}, \left\{v_n,w\right\}\in E(\Gamma)$ for $i=1,\ldots, n-1$. 
In particular, a graph $P=(\left\{v_1,\ldots, v_n\right\},\left\{\left\{v_1,v_2\right\},\left\{v_2,v_3\right\},\ldots,\left\{v_{n-1}, v_n\right\}\right\})$ is called a \emph{path from $v_1$ to $v_n$}. A subgraph $\Lambda$ is a \emph{connected component} of $\Gamma$ if $\Lambda$ is connected and is maximal with this property. Let us discuss some properties of the graph in Figure 3 in detail.

\begin{figure}[h]
\begin{center}
\begin{tikzpicture}
\draw[fill=black]  (5,0) circle (2pt);
\node at (5,-0.3) {$v_1$}; 
\draw[fill=black]  (7,0) circle (2pt);
\node at (7,-0.3) {$v_2$}; 
\draw[fill=black]  (5,2) circle (2pt);
\node at (5,1.7) {$v_4$}; 
\draw[fill=black]  (7,2) circle (2pt);
\node at (7,1.7) {$v_3$}; 
\draw (5,2)--(7,2); 
\end{tikzpicture}
\caption{$\Gamma=(\left\{v_1, v_2, v_3, v_4\right\}, \left\{\left\{v_3, v_4\right\}\right\})$}
\end{center}
\end{figure}
The subgraphs $\Delta_1=(\left\{v_3\right\},\emptyset)$ and $\Delta_2=(\left\{v_3, v_4\right\},\left\{\left\{v_3,v_4\right\}\right\})$ are  cliques in $\Gamma$ and $\Delta_2$ is a maximal clique. The graph $\Lambda=(\left\{v_3, v_4\right\},\emptyset)$ is a subgraph of $\Gamma$, but $\Lambda$ is not a full subgraph, because $v_3$ and $v_4$ are adjacent in $\Gamma$ but not in $\Lambda$. For a set $S=\left\{v_1, v_3, v_4\right\}$ the graph $\langle S\rangle$ is equal to $(\left\{v_1, v_3, v_4\right\}, \left\{\left\{v_3, v_4\right\}\right\})$. Further, $\Gamma$ has exactly three connected components. 

\subsection{Coxeter groups}
Let us recall the basic definitions and notations concerning Coxeter groups. For further details we refer to \cite{Davis}.
\begin{definition}
Let $\Gamma$ be a finite simplicial graph with an edge-labeling $\varphi:E(\Gamma)\rightarrow\mathbb{N}_{\geq 2}$.  We call such a labeled graph a \emph{Coxeter graph}. The Coxeter group $W_\Gamma$ is defined as follows:
$$W_\Gamma=\langle V(\Gamma)\mid v^2, (vw)^{\varphi(\left\{v,w\right\})} \text{ for all }v\in V(\Gamma) \text{ and whenever } \left\{v,w\right\}\in E(\Gamma)\rangle.$$
\end{definition}
Let us examine the following example of a Coxeter graph. 

\begin{figure}[h]
	\begin{center}
		\begin{tikzpicture}
			\draw[fill=black]  (0,-3) circle (1pt);
			\draw[fill=black]  (2,-3) circle (1pt);
			\draw (0,-3)--(2,-3);
			\node at (1,-2.75) {$5$}; 
			\node at (-0.5,-2.5) {$\Gamma$}; 			 
		\end{tikzpicture}
	\caption{$W_{\Gamma}=\langle v,w \mid v^2, w^2, (vw)^5\rangle\cong\mathbb{Z}/5\Z\rtimes\mathbb{Z}/2\Z$}
	\end{center}
\end{figure}

The Coxeter group $W_\Gamma$ defined by the Coxeter graph $\Gamma$ in  Figure 4 is called the dihedral group $D_5$.
This Coxeter group belongs to an infinite family of Coxeter groups: the dihedral groups $D_n$ of order $2n$. More precisely,	for $\Gamma=(\left\{v,w\right\}, \left\{\left\{v,w\right\}\right\})$ with $\varphi(\left\{v,w\right\})=n\geq3$, the corresponding Coxeter group $W_\Gamma=\langle v,w \mid v^2, w^2, (vw)^n\rangle$ is the dihedral group $D_n$, the group of isometries of the $n$-gon. Some of the algebraic properties of the dihedral groups depend on whether $n$ is even or odd. For example, 
the center of $D_n$ is trivial if $n$ is odd and is isomorphic to $\langle (vw)^{n/2}\rangle\cong\mathbb{Z}/2\Z$ if $n$ is even.
Further, the generators $v$ and $w$ are conjugate if and only if $n$ is odd. In a general Coxeter group $W_\Gamma$, generators $v$ and $w$ are conjugate if and only if there is a path in $\Gamma$ between $v$ and $w$ such that each edge in this path has an odd label, see \cite[Proposition 5.3]{BahlsIso}.

Let us describe the algebraic structure of the Coxeter groups defined by the Coxeter graphs in Figure 5. One can show that $W_{\Gamma_1}\cong{\rm Sym}(5)$,  $W_{\Gamma_2}\cong(\Z/2\Z)^4\rtimes{\rm Sym}(4)$ and $W_{\Gamma_3}\cong\Z^3\rtimes{\rm Sym}(4)$.

\begin{figure}[h]
\begin{center}
\begin{tikzpicture}
\draw[fill=black]  (0,0) circle (1pt);
\draw[fill=black]  (2,0) circle (1pt);
\draw[fill=black]  (2,2) circle (1pt);
\draw[fill=black]  (0,2) circle (1pt);
\draw (0,0)--(2,0);
\draw (0,0)--(0,2);
\draw (2,0)--(2,2);
\draw (0,2)--(2,2);
\draw (0,0)--(2,2);
\draw (2,0)--(0,2);
\node at (1,-0.25) {$3$}; 
\node at (1, 2.25) {$3$}; 
\node at (2.25, 1) {$3$}; 
\node at (-0.25,1) {$2$}; 
\node at (0.5, 0.8) {$2$}; 
\node at (1.5,0.8) {$2$}; 
\node at (-0.5,2.8) {$\Gamma_1$}; 

\draw[fill=black]  (5,0) circle (1pt);
\draw[fill=black]  (7,0) circle (1pt);
\draw[fill=black]  (5,2) circle (1pt);
\draw[fill=black]  (7,2) circle (1pt);
\draw (5,0)--(7,0);
\draw (5,0)--(5,2);
\draw (7,0)--(5,2);
\draw (5,2)--(7,2);
\draw (5,0)--(7,2);
\draw (5,0)--(5,2);
\draw (7,0)--(7,2);
\node at (6,-0.25) {$4$}; 
\node at (6, 2.25) {$3$}; 
\node at (7.25, 1) {$3$}; 
\node at (4.75,1) {$2$}; 
\node at (5.5, 0.8) {$2$}; 
\node at (6.5,0.8) {$2$}; 
\node at (4.5,2.8) {$\Gamma_2$}; 

\draw[fill=black]  (10,0) circle (1pt);
\draw[fill=black]  (12,0) circle (1pt);
\draw[fill=black]  (10,2) circle (1pt);
\draw[fill=black]  (12,2) circle (1pt);
\draw (10,0)--(12,0);
\draw (10,0)--(10,2);
\draw (12,0)--(10,2);
\draw (10,2)--(12,2);
\draw (10,0)--(12,2);
\draw (10,0)--(10,2);
\draw (12,0)--(12,2);
\node at (11,-0.25) {$3$}; 
\node at (11, 2.25) {$3$}; 
\node at (12.25, 1) {$2$}; 
\node at (9.75,1) {$2$}; 
\node at (10.5, 0.8) {$3$}; 
\node at (11.5,0.8) {$3$}; 
\node at (9.5,2.8) {$\Gamma_3$}; 
\end{tikzpicture}
\caption{}
\end{center}
\end{figure}

Given a Coxeter graph $\Gamma$, the next proposition shows that each subset $X\subseteq V(\Gamma)$ generates a Coxeter group. For a proof we refer to \cite[Theorem 4.1.6]{Davis}.
\begin{proposition}
\label{SpecialSubgroup}
Let $\Gamma$ be a Coxeter graph. For each $X\subseteq V(\Gamma)$, the subgroup which is generated by $X$ is canonically isomorphic to the Coxeter group $W_\Lambda$ where $\Lambda$ is the subgraph of $\Gamma$ generated by $X$. 
\end{proposition}

The above proposition yields to the following definition.
\begin{definition}
Let $\Gamma$ be a Coxeter graph. For a full subgraph $\Lambda\subseteq\Gamma$, the subgroup $W_\Lambda$ of $W_\Gamma$ is called \emph{special parabolic} and any conjugate of a special parabolic subgroup is called \emph{parabolic}. Note that sometimes we write for $W_{\Lambda}$ just $W_{V(\Lambda)}$.	
\end{definition}

The building blocks of a Coxeter group are irreducible special parabolic subgroups. 
\begin{definition}
A Coxeter group $W_\Gamma$ is called {\it reducible} 
if $V(\Gamma)$ can be partitioned into two non-empty disjoint subsets $V_1$ and $V_2$ such that $\left\{x,y\right\}\in E(\Gamma)$ for all $x\in V_1, y\in V_2$ and $\varphi(\left\{x,y\right\})=2$ for all $x\in V_1$ and $y\in V_2$. 
In that case $W_\Gamma=W_{V_1}\times W_{V_2}$ (\cite[Prop. 4.1.7]{Davis}).
If this sort of decomposition is not possible, we call the Coxeter group $W_\Gamma$ \emph{irreducible}.
\end{definition}

For example, the Coxeter group $W_\Gamma$ in Figure 6 is reducible.
\begin{figure}[h]
	\begin{center}
		\begin{tikzpicture}
			\draw (5,0)--(7,0);
			\draw (5,0)--(5,2);
			\draw (7,0)--(5,2);
			\draw (5,0)--(7,2);
			\draw (5,0)--(5,2);
			\draw (7,2)--(5,2);
			\draw (7,0)--(7,2);
			\draw[fill=black]  (5,0) circle (1pt);
			\draw[fill=black]  (7,0) circle (1pt);
			\draw[fill=black]  (5,2) circle (1pt);
			\draw[fill=black]  (7,2) circle (1pt);
			\node at (6,-0.25) {$3$}; 
			\node at (7.25, 1) {$2$}; 
			\node at (4.75,1) {$2$}; 
			\node at (5.5, 0.8) {$2$}; 
			\node at (6.5,0.8) {$2$}; 
			\node at (6, 2.25) {$3$};
		\end{tikzpicture}
	\caption{$W_{\Gamma}={\rm Sym}(3) \times {\rm Sym}(3)$.}
	\end{center}
\end{figure}
We note that, given a Coxeter group $W_\Gamma$ we can always decompose $W_\Gamma$ in a direct product of irreducible special subgroups.
\vspace{0.3cm}

We are now going to recall the definition of the centralizer  of a subgroup $H$ in a group $G$. The \emph{centralizer} of $H$ in $G$, denoted by $Z_G(H)$, is defined as follows: 
$$Z_G(H):=\left\{g\in G\mid gh=hg\text{ for all }h\in H\right\}.$$
The \emph{center} of $G$ is the subgroup $Z(G):=Z_G(G)$. 

Given a Coxeter group $W_\Gamma$ and an arbitrary subgroup $H$, it is in general not known what $Z_{W_\Gamma}(H)$ is. However, when $H$ is a special parabolic subgroup one can say a lot about the structure of the centralizer of $H$ in $W_\Gamma$, see  \cite{Brink}, \cite{BahlsIso}, \cite{Nuida1}. In particular, the center of $W_\Gamma$ can be calculated as follows.

\begin{lemma}
	\label{centerW}
	Let $W_\Gamma$ be a Coxeter group. Write  $W_\Gamma=W_{\Gamma_1}\times\ldots\times W_{\Gamma_n}$ as a direct product of irreducible special subgroups.  We define $J:=\left\{i\in\left\{1,\ldots, n\right\}\mid W_{\Gamma_i}\text{ is finite}\right\}$. Then $Z(W_\Gamma)=\prod\limits_{j\in J} Z(W_{\Gamma_j})$.
\end{lemma}
\begin{proof}
	The center of $W_\Gamma$ is equal to $Z(W_{\Gamma_1})\times\ldots\times Z(W_{\Gamma_n})$. Since the center of an irreducible infinite Coxeter group is trivial \cite[Thm. D.2.10]{Davis} we obtain $Z(W_\Gamma)=\prod\limits_{j\in J} Z(W_{\Gamma_j})$.   
\end{proof}

The proof of Lemma \ref{FiniteOut} is now straightforward.
\begin{proof}[Proof of Lemma \ref{FiniteOut}.]
	We write $W_\Gamma$ as a direct product of irreducible special subgroups $W_\Gamma=W_{\Gamma_1}\times W_{\Gamma_2}\times\ldots\times W_{\Gamma_n}$. Since $W_\Gamma$ is infinite by assumption, there exists $j\in\left\{1,\ldots, n\right\}$ such that $W_{\Gamma_j}$  is infinite. Note that the center of $W_{\Gamma_j}$ is trivial. We obtain 
	$${\rm Inn}(W_\Gamma)\cong W_\Gamma/Z(W_\Gamma)\cong W_{\Gamma_1}/Z(W_{\Gamma_1})\times W_{\Gamma_2}/Z(W_{\Gamma_2})\times\ldots\times W_{\Gamma_n}/Z(W_{\Gamma_n})\overset{\pi}{\twoheadrightarrow} W_{\Gamma_j},$$
	where $\pi$ is the canonical projection to the factor $W_{\Gamma_j}$. By assumption the index of ${\rm Inn}(W_\Gamma)$ in ${\rm Aut}(W_\Gamma)$ is finite, hence ${\rm Aut}(W_\Gamma)$ maps virtually onto the special subgroup $W_{\Gamma_j}$. 
\end{proof}

It is well known that for a right-angled Coxeter group $W_\Gamma$ and a special subgroup $W_\Lambda$ there exists a retraction homomorphism $r:W_\Gamma\twoheadrightarrow W_\Lambda$ such that $r(v)=v$ for all $v\in V(\Lambda)$ and $r(w)=1$ for $w\in V(\Gamma)-V(\Lambda)$. This kind of a retraction map exists also for other Coxeter groups.
\begin{proposition}(\cite[Proposition 2.1]{Gal})
\label{retraction}
Let $W_\Gamma$ be a Coxeter group and $W_\Omega$ be a special parabolic subgroup. If for every $\left\{v,w\right\}\in E(\Gamma)$ where $v\in V(\Omega)$, $w\in V(\Gamma)-V(\Omega)$ the edge label is even, then
\begin{enumerate}
\item[(i)] there exists a well-defined homomorphism $r: W_\Gamma\to W_\Omega$ such that $r_{|W_\Omega}=id$ and $r(v)=1$ for all $v\in V(\Gamma)-V(\Omega)$.
\item[(ii)] the kernel of $r$ is the normal closure of the special parabolic subgroup $W_{V(\Gamma)-V(\Omega)}$.
\end{enumerate}
\end{proposition}

Let us examine the Coxeter group $W_\Gamma$ defined by the graph $\Gamma$ in Figure 7.
\begin{figure}[h]
	\begin{center}
		\begin{tikzpicture}
			\draw[fill=black]  (0,0) circle (1pt);
			\draw[fill=black]  (2,0) circle (1pt);
			\draw[fill=black]  (2,2) circle (1pt);
			\draw[fill=black]  (0,2) circle (1pt);
			\draw (0,0)--(2,0);
			\node at (1, -0.2) {$4$}; 
			\draw (2,0)--(2,2);
			\node at (2.2, 1) {$3$};
			\draw (2,2)--(0,2);
			\node at (1, 2.2) {$4$};
			\draw (0,2)--(0,0);
			\node at (-0.2, 1) {$3$};
			\draw[fill=black]  (-2,1) circle (1pt);
			\draw (-2,1)--(0,2);
			\node at (-1, 1.8) {$3$};
			\draw (-2,1)--(0,0);
			\node at (-1, 0.2) {$3$};
			\draw[fill=black]  (4,1) circle (1pt);
			\draw (4,1)--(2,2);
			\node at (3, 1.8) {$2$};
			\draw (4,1)--(2,0);
			\node at (3, 0.2) {$2$};
			\node at (-1.5, 2.5) {$\Gamma$}; 
			\node at (-2.3, 1) {$v_1$};
			\node at (0, -0.3) {$v_2$};
			\node at (0, 2.3) {$v_3$};
			\node at (4.3, 1) {$v_6$};
			\node at (2, -0.3) {$v_4$};
			\node at (2, 2.3) {$v_5$};
		\end{tikzpicture}
	\caption{}
	\end{center}
\end{figure}
Applying Proposition \ref{retraction} we know that there are three retractions: $W_\Gamma\twoheadrightarrow W_{\left\{v_1,v_2, v_3\right\}}$, $W_\Gamma\twoheadrightarrow W_{\left\{v_4,v_5\right\}}$ and $W_\Gamma\twoheadrightarrow W_{\left\{v_6\right\}}$. 

\subsection{Characteristic subgroups}
For the proofs of Theorems \ref{TwoConnectedComponents} and \ref{Even} the notion of characteristic subgroups with respect to a subgroup of the automorphism group will be needed. We refer to \cite{Leder} for more information on these subgroups.  

\begin{definition}
	Let $G$ be a group and $H\subseteq{\rm Aut}(G)$ be a subgroup. A normal subgroup $N\trianglelefteq G$ is called \emph{characteristic with respect to $H$}, if for all $f\in H$ we have $f(N)=N$.  
\end{definition}
Let $N\trianglelefteq G$ be characteristic with respect to $H\subseteq{\rm Aut}(G)$, then the map $H\to{\rm Aut}(G/N)$ given by $f\mapsto [gN\mapsto f(g)N]$ is a well-defined group homomorphism, see \cite[Lemma 7.4]{Leder}.

\begin{definition}
Let $W_\Gamma$ be a Coxeter group. We define \emph{the set of special automorphisms} ${\rm Spe}(W_\Gamma)\subseteq{\rm Aut}(W_\Gamma)$ as follows:
$f\in{\rm Spe}(W_\Gamma)$ if and only if for every $w\in W$ with $w^2=1$ there
exists $g\in W_\Gamma$  such that $f(w)=gwg^{-1}$.
\end{definition}
Clearly, ${\rm Spe}(W_\Gamma)$ is a group and the subgroup consisting of inner automorphisms ${\rm Inn}(W_\Gamma)$ is contained in ${\rm Spe}(W_\Gamma)$.

\begin{lemma}
	\label{retractionSpe}
	Let $W_\Gamma$ be a Coxeter group and $W_\Omega$  a special parabolic subgroup. 
	Assume that for every $\left\{v,w\right\}\in E(\Gamma)$ where $v\in V(\Omega)$, $w\in V(\Gamma)-V(\Omega)$ the edge label is even.  
	
	Then $K:={\rm ker}(r)$, where $r\colon W_\Gamma\twoheadrightarrow W_\Omega$ is the retraction map onto $W_\Omega$, is a characteristic subgroup with respect to ${\rm Spe}(W_\Gamma)$. Further, the image of the canonical map 
	$$\Phi\colon{\rm Spe}(W_\Gamma)\to{\rm Aut}(W_\Gamma/K)\cong {\rm Aut}(W_\Omega)$$
	where the isomorphism ${\rm Aut}(W_\Gamma/K)\cong {\rm Aut}(W_\Omega)$ is induced by $r$, contains ${\rm Inn}(W_\Omega)$. 
\end{lemma}
\begin{proof}
Let $f$ be a special automorphism of $W_\Gamma$. By Proposition \ref{retraction} the kernel of $r$ is generated by the set $\left\{xvx^{-1}\mid v\in V(\Gamma)-V(\Omega), x\in W_{\Gamma}\right\}$.
For $x\in W_{\Gamma}$ and $v\in V(\Gamma)-V(\Omega)$ there exists $a\in W_\Gamma$ such that
$$f(xvx^{-1})=f(x)ava^{-1}f(x^{-1})\in K.$$
Thus $f(K)\subseteq K$. 
This shows that for all $g\in{\rm Spe}(W_\Gamma)$ we have $g(K)\subseteq K$. In particular, we have $f^{-1}(K)\subseteq K$ which yields $K\subseteq f(K)$ and hence $f(K)=K$. Thus, $K$ is characteristic with respect to ${\rm Spe}(W_\Gamma)$.

Let $\iota\colon W_\Omega\rightarrow W_\Gamma/K$ be the induced isomorphism by $r$. Then $\iota$ induces an isomorphism as follows
$$\Psi\colon{\rm  Aut}(W_\Gamma/K)\to{\rm Aut}(W_\Omega)\text{ } \ f\mapsto \iota^{-1}\circ\Phi(f)\circ\iota.$$
Let $\rho\in{\rm Inn}(W_\Gamma)$ be a conjugation by $w\in W_\Omega$. Our goal is to show that $\rho_{|W_\Omega}\in{\rm Inn}(W_\Omega)$ has a preimage under $\Psi\circ\Phi$.
We consider the map
$$\Psi\circ\Phi\colon{\rm Spe}(W_\Gamma)\to{\rm Aut}(W_\Omega)$$
By definition of $\Phi$ we have for $x\in W_\Gamma$:
$\Phi(\rho)(xK)=wxw^{-1}K$.
Thus for $x\in W_\Omega$ we obtain
$$(\Psi\circ\Phi)(\rho)(x)=(\iota^{-1}\circ\Phi(\rho)\circ\iota) (x)=(\iota^{-1}\circ\Phi(\rho))(xK)=\iota^{-1}(wxw^{-1}K)\overset{wxw^{-1}\in W_\Omega}{=}wxw^{-1}$$
Hence $(\Psi\circ\Phi)(\rho)=\rho_{|W_\Omega}$.
That shows that the group ${\rm Inn}(W_\Omega)$ is contained in the image of $\Psi\circ\Phi$.
\end{proof}

\section{Proof of Theorem \ref{Even}}
\begin{proof}[Proof of Theorem \ref{Even}.]
	It was proven in \cite{Richardson} that the cardinality of conjugacy classes of subgroups of order $2$ in $W_\Gamma$ is finite. The group ${\rm Aut}(W_\Gamma)$ acts on these conjugacy classes in the canonical way. The kernel of this action is ${\rm Spe}(W_\Gamma)$. Note that ${\rm Spe}(W_\Gamma)$ has finite index in ${\rm Aut}(W_\Gamma)$. 
	
	By assumption there exists a full subgraph $\Omega\subseteq\Gamma$ such that $W_\Omega$ is infinite and ${\rm Out}(W_\Omega)$ is finite. Further, for every edge $\left\{v,w\right\}\in E(\Gamma)$ where $v\in V(\Gamma)-V(\Omega)$, $w\in V(\Omega)$, the edge label is even.
	
	We claim that ${\rm Spe}(W_\Gamma)$ surjects virtually onto ${\rm Inn}(W_\Omega)$. Let 
	$r:W_\Gamma\twoheadrightarrow W_\Omega$
	be the retraction map which exists by Proposition \ref{retraction}. By Lemma \ref{retractionSpe} the kernel $K:={\rm ker}(r)$ is characteristic with respect to ${\rm Spe}(W_\Gamma)$. Hence we get the following canonical homomorphism 
	$$\Psi:{\rm Spe}(W_\Gamma)\rightarrow{\rm Aut}(W_\Gamma/K)\cong {\rm Aut}(W_\Omega).$$
	Further, by Lemma \ref{retractionSpe} the image of $\Psi$ contains ${\rm Inn}(W_\Omega)$. Since ${\rm Out}(W_\Omega)$ is finite by assumption, the group $\Psi^{-1}({\rm Inn}(W_\Omega))$ has finite index in ${\rm Spe}(W_\Gamma)$ and surjects onto ${\rm Inn}(W_\Omega)$. The proof of Lemma \ref{FiniteOut} shows that ${\rm Inn}(W_\Omega)$ maps onto an infinite special subgroup of $W_\Omega$. Hence ${\rm Aut}(W_\Gamma)$
    maps virtually onto an infinite special subgroup of $W_\Omega$.
\end{proof}


\section{Proof of Theorem \ref{FiniteOut}}

The crucial piece in the proof of Theorem \ref{FiniteOut} is the following result.
\begin{theorem}(\cite[Theorem 28]{MihalikTschantz})
	\label{CW}
	Let $\Gamma$ be a Coxeter graph and $\Delta_1,\ldots, \Delta_n$ be all maximal cliques of $\Gamma$. Let $C(W_\Gamma)$ be the subgroup of all $f\in {\rm Aut}(W_\Gamma)$ such that for $i\in\left\{1,\ldots, n\right\}$ there exists an element $w_{i,f}\in W_\Gamma$ so that: $$f(x)=w_{i,f}\cdot x\cdot {w_{i,f}}^{-1}\text{\ for all\ }
	x\in W_{\Delta_i}.$$
	The subgroup $C(W_\Gamma)$ has finite index in ${\rm Aut}(W_\Gamma)$.
\end{theorem}

A natural question is when the subgroup $C(W_\Gamma)$ is equal to the inner automorphism group of $W_\Gamma$. If $\Gamma$ is not connected, then in most cases $C(W_\Gamma)$ is much larger than ${\rm Inn}(W_\Gamma)$ as the following observation shows: Let $\Gamma_1, \ldots, \Gamma_n$ be the connected components of $\Gamma$ and $W_{\Gamma_1}, \ldots, W_{\Gamma_n}$ the corresponding special subgroups. Then $W_\Gamma=W_{\Gamma_1}*\ldots* W_{\Gamma_n}$. We assume that $n\geq 2$. Let $a$ be a non-central element in $W_{\Gamma_1}$, more precisely $a\in W_{\Gamma_1}$ and $a\notin Z(W_{\Gamma_1})$. We define $f_{12}:W_\Gamma \rightarrow W_\Gamma$ as follows: 
$f_{12}(v)=v$ for all $v\in V(\Gamma)-V(\Gamma_2)$ and $f_{12}(w)=awa^{-1}$ for all $w\in V(\Gamma_2)$. The map $f_{12}$ induces naturally an automorphism of $W_\Gamma$. Moreover $f_{12}$ is contained in $C(W_\Gamma)$. Since the element $a$ does not commute with $W_{\Gamma_1}$ the automorphism $f_{12}$ is not inner. 

\begin{lemma}
	\label{CWInn}
	Let $\Gamma$ be a connected Coxeter graph. We denote by $\Delta_1, \ldots, \Delta_n$ its maximal cliques and by $\Omega_1,\ldots, \Omega_m$ the nonempty intersections of these maximal cliques where the intersection is taken over at least two maximal cliques. 
	
	If the centralizer of the special parabolic subgroup $W_{\Omega_i}$ is trivial for all $i=1,\ldots, m$, then the subgroup $C(W_\Gamma)={\rm Inn}(W_\Gamma)$. In particular, ${\rm Inn}(W_\Gamma)$  has finite index in ${\rm Aut}(W_\Gamma)$.	
\end{lemma}
\begin{proof}
	Let $f$ be in $C(W_\Gamma)$. Our goal is to show that $f$ is an inner automorphism of $W_\Gamma$. Clearly, it is sufficient to show that there exists $a\in W_\Gamma$ such that for all $v\in V(\Gamma)$ we have $f(v)=ava^{-1}$. 
	
	By the definition of $C(W_\Gamma)$ we know that there exists $w_{1,f}\in W_\Gamma$ such that for all $x\in W_{\Delta_1}$ we have $f(x)=w_{1,f}\cdot x\cdot {w_{i,f}}^{-1}$. If $\Gamma$ itself is a clique, then we are done. Otherwise, since $\Gamma$ is connected there exists a vertex $v\in V(\Gamma)-V(\Delta_1)$ and $w\in V(\Delta_1)$ such that $\left\{v,w\right\}\in E(\Gamma)$. Further, there exists $\Delta_j$ such that $v,w\in V(\Delta_j)$. Hence, the intersection of $\Delta_1$ and $\Delta_j$ is non-empty.
	Again, by definition of $C(W_\Gamma)$ there exists $w_{j,f}\in W_\Gamma$ such that $f(x)=w_{j,f}\cdot x\cdot {w_{j,f}}^{-1}$ for all $x\in W_{\Delta_j}$.
	
	For $x\in W_{\Delta_1}\cap W_{\Delta_j}=W_{\Delta_1\cap\Delta_j}$ we have
	$$f(x)=w_{1,f}\cdot x\cdot {w_{1,f}}^{-1}=w_{j,f}\cdot x\cdot {w_{j,f}}^{-1}$$
	Hence the element ${w_{j,f}}^{-1}w_{1,f}\in Z_{W_\Gamma}(W_{\Delta_1\cap\Delta_j})$. By assumption, this centralizer is trivial, thus $w_{1,f}=w_{j,f}$. 
	
	In particular, the above proof strategy shows that for $k\in\left\{1,\ldots, n\right\}$ where $\Delta_k\cap \Delta_1\neq\emptyset$ we have 
	$$f(v)=w_{1,f}\cdot v\cdot {w_{1,f}}^{-1}\text{\ for all \ }v\in V(\Delta_k).$$
		
	Let $J:=\left\{k\in\left\{1,\ldots, n\right\}\mid \Delta_k\cap\Delta_1\neq\emptyset\right\}$. If $V(\Gamma)= \bigcup_{k\in J} V(\Delta_k)$, then we have proved that $f$ is a global conjugation by $w_{1,f}$. 
		Otherwise the set $V(\Gamma)-\bigcup_{k\in J} V(\Delta_k)$ is non-empty. Since $\Gamma$ is connected, there exist $v\in V(\Gamma)-\bigcup_{k\in J} V(\Delta_k)$ and  $w\in\bigcup_{k\in J} V(\Delta_k)$ such that $\left\{v,w\right\}\in E(\Gamma)$. We choose $\Delta_l$ such that $v,w\in V(\Delta_l)$. Hence there exists $k\in J$ such that $\Delta_l\cap\Delta_k\neq\emptyset$.
	Now we proceed as above to show that $f(x)=w_{1,f}\cdot x\cdot {w_{1,f}}^{-1}$ for all $x\in \Delta_l$.
	
	After applying the above proof strategy finitely many times, we see that $f$ is a global conjugation by $w_{1,f}$.
\end{proof}

Now we turn to the proof of Theorem \ref{InfiniteComplete}.
\begin{proof}[Proof of Theorem \ref{InfiniteComplete}]
	Let $W_\Gamma$ be an infinite Coxeter group that satisfies the assumptions of Theorem \ref{FiniteOut}. By Theorem \ref{CW} we know that the subgroup $C(W_\Gamma)$ is of finite index in ${\rm Aut}(W_\Gamma)$ and by Lemma \ref{CWInn} we have $C(W_\Gamma)={\rm Inn}(W_\Gamma)$. Hence, ${\rm Inn}(W_\Gamma)$ has finite index in ${\rm Aut}(W_\Gamma)$. Since $W_\Gamma$ is infinite, Lemma \ref{FiniteOut} implies that ${\rm Aut}(W_\Gamma)$ maps virtually onto some infinite irreducible special subgroup of $W_\Gamma$.
\end{proof}	

\section{Proof of Theorem \ref{TwoConnectedComponents}}

We begin this section by recalling the Kurosh subgroup theorem \cite{Kurosh}. Let $G$ be a free product of groups $G_1,\ldots, G_n$. The Kurosh subgroup theorem describes the algebraic structure of subgroups in $G=G_1*\ldots*G_n$. Let $H\subseteq G$ be a subgroup. Then $H$ decomposes as a free product as follows
$H=F*H_1\ldots*H_m$
where $F$ is a free group and for $i=1,\ldots,m$ the group $H_i$ is a conjugate of a subgroup of some $G_j$, $j\in\left\{1,\ldots, n\right\}$.

\begin{definition}
	Let $G$ be a group. The group $G$ is called \emph{freely indecomposable} if $G$ can not be written as a non-trivial free product of groups.
\end{definition}

Given a free product $G_1*G_2$ and a subgroup $H\subseteq G_1*G_2$. the Kurosh subgroup theorem implies that if $H$ is freely indecomposable, then $H$ is infinite cyclic or is contained in a conjugate of $G_1$ or $G_2$.

Given a finitely generated group $G$ and two free product decompostions of $G$ in freely indecomposable groups, a natural question is how these decompositions differ. The answer to this question is given by Grushko's decomposition theorem \cite{Grushko}. We recall this result in the special case where $G$ is a Coxeter group, see also \cite{Mc}. 
\begin{proposition}
	\label{Kurosh}
Let $W_\Gamma$ be a Coxeter group.
\begin{enumerate}
	\item If $\Gamma$ is connected, then $W_\Gamma$ is freely indecomposable (\cite{BahlsFree}).
	\item Let $\Gamma_1, \ldots, \Gamma_n$ be the connected components of $\Gamma$. Then $W_\Gamma=W_{\Gamma_1}*\ldots*W_{\Gamma_n}$ is a free product of freely indecomposable special subgroups.
	If $W_\Gamma$ is a free product of freely indecomposable groups $G_1,\ldots, G_m$, then $n=m$ and there exists $\sigma\in Sym(n)$ such that $G_{\sigma(i)}$ is conjugate to $W_{\Gamma_i}$.
	\item For $f\in{\rm Aut}(W_\Gamma)$ we have $W_\Gamma=f(W_{\Gamma_1})*\ldots*f(W_{\Gamma_n})$. 
\end{enumerate} 
\end{proposition}

Now we have all the ingredients for the proof of Theorem \ref{TwoConnectedComponents}.
\begin{proof}[Proof of Theorem \ref{TwoConnectedComponents}]
	Let $\Gamma$ be a Coxeter graph and $\Gamma_1,\ldots, \Gamma_n$ be the connected components of $\Gamma$. We assume that $n\geq 2$. We define $[W_{\Gamma_i}]:=\left\{wW_{\Gamma_i}w^{-1}\mid w\in W_\Gamma\right\}$ for $i=1,\ldots, n$ and $\mathcal{C}:=\left\{[W_{\Gamma_1}],\ldots, [W_{\Gamma_n}]\right\}$. Proposition \ref{Kurosh} implies that the map
	\begin{align*}
		\Phi:{\rm Aut}(W_\Gamma)&\rightarrow{\rm Sym}(\mathcal{C})\\
		f&\mapsto\Phi(f):[W_{\Gamma_i}]\mapsto[f(W_{\Gamma_i})]
	\end{align*}
	is a well-defined homomorphism. 
	
	Let $i,j\in\left\{1,\ldots,n\right\}, i\neq j$. Let us examine the natural projection $\pi\colon W_\Gamma\twoheadrightarrow W_{\Gamma_i}*W_{\Gamma_j}$. Note that the kernel of $\pi$ is the normal closure of the subgroup $\langle W_{\Gamma_k}\mid k\in\left\{1,\ldots, n\right\}, k\neq i,j\rangle$. We claim that ${\rm ker}(\pi)$ is characteristic with respect to  ${\rm ker}(\Phi)\subseteq {\rm Aut}(W_\Gamma)$. Let $f\in{\rm ker}(\Phi)$. Then for $W_{\Gamma_k}, k\neq i,j$ we have $f(W_{\Gamma_k})=w_kW_{\Gamma_k}w_k^{-1}$. Since the kernel of $\pi$ is generated by conjugates of elements in these free factors we obtain $f({\rm ker}(\pi))={\rm ker}(\pi)$.

	Hence we get a group homomorphism
	$$\Psi\colon{\rm ker}(\Phi)\to{\rm Aut}(W_\Gamma/{\rm ker}(\pi))\cong {\rm Aut}(W_{\Gamma_i}*W_{\Gamma_j})$$

  It is not hard to see that  the group ${\rm Inn}(W_{\Gamma_i}*W_{\Gamma_j})$ is contained in $\Psi({\rm ker}(\Phi))$.
  
  Now let $\rho\colon W_{\Gamma_i}*W_{\Gamma_j}\twoheadrightarrow W_{\Gamma_i}^{ab}*W_{\Gamma_j}^{ab}$ be group homomorphism induced by the abelianization $W_{\Gamma_i}\twoheadrightarrow W_{\Gamma_i}^{ab}$. 
  We claim that ${\rm ker}(\rho)$ is characteristic with respect to ${\rm im}({\rm ker}(\Phi))$. Since the kernel of $\rho$ is generated by conjugates of $ghg^{-1}h^{-1}$ where $g,h\in W_{\Gamma_i}$ and $xyx^{-1}y^{-1}$ where $x,y\in W_{\Gamma_j}$ we only have to consider what happents to these elements under $f\in{\rm im}({\rm ker}(\Phi))$. A straightforward calculation shows that ${\rm ker}(\rho)$ is characteristic with respect to ${\rm im}({\rm ker}(\Phi))$.
  
  Thus we obtain a group homomorphism
  $$\Sigma\circ\Psi\colon{\rm ker}(\Phi)\to{\rm im}({\rm ker}(\Phi))\to{\rm Aut}(W_{\Gamma_i}^{ab}*W_{\Gamma_j}^{ab}) $$
  
  Note that the subgroup ${\rm Inn}(W_{\Gamma_i}^{ab}*W_{\Gamma_j}^{ab})$ is contained in $\Sigma(\Psi({\rm ker}(\Phi)))$.
  By \cite[Corollary 3.21]{GenevoisMartin} the subgroup ${\rm Inn}(W^{ab}_{\Gamma_i}*W^{ab}_{\Gamma_j})$ has finite index in ${\rm Aut}(W_{\Gamma_i}^{ab}*W_{\Gamma_j}^{ab})$. Thus the subgroup $H:=(\Sigma\circ\Psi)^{-1}({\rm Inn}(W^{ab}_{\Gamma_i}*W^{ab}_{\Gamma_j}))$ has finite index in ${\rm ker}(\Phi)$ and we have an epimorphism
  $$H\twoheadrightarrow{\rm Inn}(W_{\Gamma_i}^{ab}*W_{\Gamma_j}^{ab})\cong W_{\Gamma_i}^{ab}*W^{ab}_{\Gamma_j}\cong\Z/2\Z^k*\Z/2\Z^l.$$ 
  This shows that ${\rm Aut}(W_\Gamma)$ virtually maps onto the infinite right-angled Coxeter group $\Z/2\Z^k*\Z/2\Z^l$.
\end{proof}

\begin{figure}[h]
	\begin{center}
		\begin{tikzpicture}
			\draw[fill=black]  (0,0) circle (1pt);
			\draw[fill=black]  (2,0) circle (1pt);
			\draw[fill=black]  (-2,0) circle (1pt);
			\draw[fill=black]  (0,2) circle (1pt);
			\draw[fill=black]  (0,-2) circle (1pt);
			\draw (-2,0)--(0,0);
			\node at (-1, 0.2) {$3$};
			\draw (0,0)--(2,0);
			\node at (1, 0.2) {$2$};
			\draw (0,0)--(0,2);
			\node at (-0.2, 1) {$2$};
			\draw (0,0)--(0,-2);
			\node at (0.2, -1) {$2$};
			\draw (-2,0)--(0,2);
			\node at (-1, 1.3) {$2$};
			\draw (0,2)--(2,0);
			\node at (1, 1.3) {$3$};
			\draw (-2,0)--(0,-2);
			\node at (-1, -1.3) {$2$};
			\draw (0,-2)--(2,0);
			\node at (1, -1.3) {$3$};
			\draw[fill=black]  (6,0) circle (1pt);
			\draw[fill=black]  (8,0) circle (1pt);
			\draw[fill=black]  (4,0) circle (1pt);
			\draw[fill=black]  (6,2) circle (1pt);
			\draw[fill=black]  (6,-2) circle (1pt);
			\draw (4,0)--(6,0);
			\node at (5, 0.2) {$3$};
			\draw (6,0)--(8,0);
			\node at (7, 0.2) {$2$};
			\draw (6,0)--(6,2);
			\node at (5.8, 1) {$2$};
			\draw (6,0)--(6,-2);
			\node at (6.2, -1) {$2$};
			\draw (4,0)--(6,2);
			\node at (5, 1.3) {$2$};
			\draw (6,2)--(8,0);
			\node at (7, 1.3) {$3$};
			\draw (4,0)--(6,-2);
			\node at (5, -1.3) {$2$};
			\draw (6,-2)--(8,0);
			\node at (7, -1.3) {$3$};
			\node at (-2, 1.8) {$\Gamma$};
		\end{tikzpicture}
	\end{center}
\caption{}
\end{figure}
Let us calculate an example. The Coxeter graph $\Gamma$ in Figure 8 has exactly two connected components $\Gamma_1$ and $\Gamma_2$. The abelianization of $W_{\Gamma_i}$ is isomorphic to $\Z/2\Z\times \Z/2\Z$ for $i=1,2$, see see \cite[Fact 3.16]{Bernhard}. Thus ${\rm Aut}(W_\Gamma)$ projects virtually onto $(\Z/2\Z\times \Z/2\Z)*(\Z/2\Z\times\Z/2\Z)$.

\begin{proof}[Proof of Corollary \ref{CorEven}]
Let $W_\Gamma$ be a Coxeter group such that there exist two non-adjacent even vertices $v, w\in V(\Gamma)$.  Then we know that the subgroup generated by $v$ and $w$ in $W_\Gamma$ is isomorphic to the infinite dihedral group whose outer automorphism group known to be finite. Thus by Theorem \ref{TwoConnectedComponents}  the automorphism group ${\rm Aut}(W_\Gamma)$  virtually surjects onto $\Z/2\Z*\Z/2\Z$.

In particular, if $W_\Gamma$ is an even Coxeter group and there exist two non-adjacent vertices, then  ${\rm Aut}(W_\Gamma)$ virtually surjects onto $\Z/2\Z*\Z/2\Z$. Further, if $W_\Gamma$ is an infinite even Coxeter group where $\Gamma$ is complete, then we know by \cite{HowlettRowleyTaylor} that ${\rm Out}(W_\Gamma)$ is finite and thus by Lemma \ref{FiniteOut} the group ${\rm Aut}(W_\Gamma)$ virtually surjects onto an infinite special subgroup of $W_\Gamma$.		
\end{proof}



\begin{thebibliography}{GHKR10}
 
\bibitem[Bah05]{BahlsIso}
P.~Bahls, 
The isomorphism problem in Coxeter groups. 
Imperial College Press, London, 2005.

\bibitem[Bah06]{BahlsFree}
P. Bahls, 
Automorphisms of Coxeter groups. 
Trans. Amer. Math. Soc. 358 (2006), no. 4, 1781--1796. 

\bibitem[BM05]{Bahls} 
P.~Bahls; M.~Mihalik, 
Reflection independence in even Coxeter groups. 
Geom. Dedicata 110 (2005), 63--80. 

\bibitem[MPS22]{Bernhard}
B. Mühlherr; G. Paolini; S. Shelah, 
First-order aspects of Coxeter groups. 
J. Algebra 595 (2022), 297--346.

\bibitem[BHV08]{BekkaHarpeValette} 
B.~Bekka; P.~de la Harpe; A.~Valette, 
Kazhdan's property (T). 
New Mathematical Monographs, 11. Cambridge University Press, Cambridge, 2008.

\bibitem[BJS88]{BJS} 
M.~Bozejko; T.~Januszkiewicz; R.~J.~Spatzier, 
Infinite Coxeter groups do not have Kazhdan's property. 
J. Operator Theory 19 (1988), no. 1, 63--67.

\bibitem[Bri96]{Brink}
B. Brink, 
On centralizers of reflections in Coxeter groups.
Bull. London Math. Soc. 28 (1996), 465--470.

\bibitem[Dav08]{Davis} 
M.~W.~Davis, 
The geometry and topology of Coxeter groups. 
London Mathematical Society Monographs Series, 32. Princeton University Press, Princeton, NJ, 2008. 

\bibitem[Die18]{Diestel} 
R.~Diestel, Graph theory. Graduate Texts in Mathematics, 173. Springer, Berlin, 2018.

\bibitem[Fra01]{Franzsen} 
W.~N.~Franzsen, 
Automorphisms of Coxeter Groups. PhD thesis, University of Sydney, (2001).

\bibitem[Gal05]{Gal} 
S.~R.~Gal, On normal subgroups of Coxeter groups generated by standard parabolic subgroups.
Geom. Dedicata 115 (2005), 65--78. 

\bibitem[GM19]{GenevoisMartin}
A.~Genevois; A.~Martin, 
Automorphisms of graph products of groups from a geometric perspective. Proceedings of the London Mathematical Society, 119 (6), 2019, 1745--1779. 

\bibitem[GV20]{GenevoisVarghese} 
A.~Genevois; O.~Varghese,
Conjugating automorphisms of graph products: Kazhdan's property (T) and SQ-universality. 
Bulletin of the Australian Mathematical Society  101 (2020), no. 2, 272--282. 

\bibitem[Gon97]{Gonciulea} 
C.~Gonciulea, 
Infinite Coxeter groups virtually surject onto $\Z$. 
Comment. Math. Helv. 72 (1997), no. 2, 257--265.

\bibitem[Gru40]{Grushko}
I.~A.~Grushko, 
On generators of a free product of groups. 
Matem. Sbornik N. S.
8 (1940) 169--182.

\bibitem[HRT97]{HowlettRowleyTaylor} 
R.~B.~Howlett; P.~J.~Rowley; D.~E.~Taylor,
On outer automorphism groups of Coxeter groups.
Manuscripta Math. 93 (1997), no. 4, 499--513. 

\bibitem[Hum90]{Humphreys} 
J.~E.~Humphreys, 
Reflection groups and Coxeter groups. 
Cambridge Studies in Advanced Mathematics, 29. Cambridge University Press, Cambridge, 1990. 

\bibitem[KKN21]{KalubaKielakNowak} 
M.  Kaluba;  D.  Kielak;  P.  Nowak,  On  property  (T)  for  ${\rm Aut}(F_n)$ and  ${\rm SL}_n(\mathbb{Z})$.
Ann. of Math. (2) 193 (2021), no. 2, 539--562.

\bibitem[KNO19]{KalubaNowakOzawa} 
M. Kaluba; P. Nowak; N. Ozawa, 
${\rm Aut}(F_5)$ has property (T). 
Math. Ann. 375 (2019), no. 3-4, 1169--1191.

\bibitem[Kur34]{Kurosh}
A. Kurosh, 
Die Untergruppen der freien Produkte von beliebigen Gruppen. Mathematische Annalen, vol. 109 (1934), 647--660.

\bibitem[Led20]{Leder}
N.~J.~Leder,
Automorphism groups of graph products and Serre’s property FA. 
PhD. thesis, M unster
University, (2020).

\bibitem[Nit20]{Nitsche} 
M. Nitsche, Computer proofs for Property (T), and SDP duality. ArXiv:2009.05134, 2020. 

\bibitem[Nui11]{Nuida1} 
K.~Nuida, 
On centralizers of parabolic subgroups in Coxeter groups. 
J. Group Theory 14 (2011), no. 6, 891--930. 

\bibitem[MM96]{Mc}
D. McCullough; A. Miller, 
Symmetric automorphisms of free products. 
Mem. Amer. Math. Soc. 122 (1996), no. 582.

\bibitem[MT09]{MihalikTschantz} 
M.~Mihalik; S.~Tschantz, Visual decompositions of Coxeter groups. 
Groups Geom. Dyn. 3 (2009), no. 1, 173--198.

\bibitem[Ric82]{Richardson} 
R.~W.~Richardson, 
Conjugacy classes of involutions in Coxeter groups. 
Bull. Austral. Math. Soc. 26 (1982), no. 1, 1--15.

\bibitem[SS19]{SaleSusse} 
A.~Sale; T.~Susse, 
Outer automorphism groups of right-angled Coxeter groups are either large or virtually abelian. 
Trans. Amer. Math. Soc. 372 (2019), no. 11, 7785--7803. 


\bibitem[Var21]{Varghese} 
O.~Varghese, 
The automorphism group of the universal Coxeter group. 
Expo. Math. 39 (2021), no. 1, 129--136. 
\end{thebibliography}
\end{document}